\documentclass[a4paper,10pt,reqno]{amsart}

\usepackage{amsmath,amssymb}
\usepackage{graphicx}
\usepackage{subfigure,float}

\newtheorem{theorem}{Theorem}[section]
\newtheorem*{theorem*}{Theorem}
\newtheorem{lemma}[theorem]{Lemma}
\newtheorem{proposition}[theorem]{Proposition}
\newtheorem*{proposition*}{Proposition}
\newtheorem{corollary}[theorem]{Corollary}
\theoremstyle{definition}
\newtheorem{definition}[theorem]{Definition}

\newtheorem{remark}[theorem]{Remark}

\newcommand{\CC}{{\mathbb{C}}}
\newcommand{\RR}{{\mathbb{R}}}

\newcommand{\NN}{{\mathbb{N}}}

\newcommand{\id}{\ensuremath{\mathrm{id}}} 
\newcommand{\dint}{\,\mathrm{d}} 
\newcommand{\enmo}[1]{\ensuremath{{\mathrm{End}\!\left( {#1} \right)}}} 
\newcommand{\aut}[1]{\ensuremath{{\mathrm{Aut}\left( {#1} \right)}}} 
\newcommand{\smooth}{{\ensuremath{\mathcal{C}^\infty}}} 
\newcommand{\smoothone}{{\ensuremath{\mathcal{C}^1}}} 
\newcommand{\smoothtwo}{{\ensuremath{\mathcal{C}^2}}} 
\newcommand{\smoothany}{\mathcal{C}} 
\newcommand{\boundary}{b} 
\newcommand{\udisk}{{\mathbb{D}}} 

\title{Riemann surfaces in Stein manifolds with density property}
\author{Rafael B. Andrist}
\address{Mathematisches Institut, Universit\"at Bern, Sidlerstr. 5, 3012 Bern, Switzerland}
\email{rafael.andrist@math.unibe.ch}

\author{Erlend Forn\ae ss Wold}
\address{Matematisk Institutt, Universitetet i Oslo, Postboks 1053 Blindern, 0316 Oslo}
\email{erlendfw@math.uio.no}
\begin{document}
\begin{abstract}
It is shown that any open Riemann surface can be immersed in any Stein manifold with (volume) density property and of dimension at least $2$, if the manifold possesses an exhaustion with holomorphically convex compacts such that their complement is connected. The immersion can be made into an embedding if the dimension is at least $3$. As an application, it is shown that Stein manifolds with (volume) density property and of dimension at least $3$,  are characterized among all other complex manifolds by their semigroup of holomorphic endomorphisms.
\end{abstract}
\maketitle
\tableofcontents

\section{Introduction}

An open Riemann surface always admits a proper holomorphic embedding into $\CC^3$ and an proper holomorphic immersion into $\CC^2$.
In this paper we generalize these results to embeddings and immersions of open Riemann surfaces into Stein manifolds with the density or volume density property. Our main result is the following:

\begin{theorem*}[\ref{thm:embedsurface}]
Let $X$ be a Stein manifold with density property or with volume density property and $\mathcal{R}$ an open Riemann surface.
If $\mathcal{R} \not \equiv \CC$, then assume further there exists a strongly plurisubharmonic exhaustion function $\tau$ of $X$ with increasing compact sublevel sets $K_j = \tau^{-1}([-\infty, M_j])$, $M_{j+1} > M_j > 0$, $j \in \NN$, $\displaystyle \lim_{j \to \infty} M_j = \infty$, such that $X \setminus K_j$ is connected for all $j \in \NN$.
\begin{itemize}
 \item[(a)] If $\dim X \geq 3$ then there is a proper holomorphic embedding $\mathcal{R} \hookrightarrow X$
 \item[(b)] If $\dim X = 2$ then there is a proper holomorphic immersion $\mathcal{R} \to X$.
\end{itemize}
\end{theorem*}
We can also avoid the assumption that $X\setminus K_j$ is connected if $\mathcal R$
admits a subharmonic exhaustion function with finitely many critical points.

\medskip

Our main motivation is the recent work of the first author \cite{Andrist}, who has shown that 
Stein manifolds are characterized by their endomorphism semigroups as long as they admit a proper holomorphic embedding of the complex line $\CC$.
As a corollary to our main theorem we thereby obtain:

\begin{theorem*}[\ref{thm:densconjugate}]
Let $X$ and $Y$ be complex manifolds and $\Phi : End(X) \to End(Y)$ an epimorphism of semigroups of holomorphic endomorphisms.
If $X$ is a Stein manifold with density or volume density property and of dimension at least $3$, then there exists a unique $\varphi : X \to Y$ which is is either biholomorphic or antibiholomorphic and such that $\Phi(f) = \varphi \circ f \circ \varphi^{-1}$.
\end{theorem*}

Secondly, our work generalizes recent work of Drinovec-Drnov\v{s}ek and  Forstneri\v{c}
\cite{DrinovecForstneric} who have proven that \emph{bordered} Riemann Surfaces 
immerse properly into Stein manifolds. Note that, due to hyperbolicity, the complex 
plane does not embed in all Stein manifolds, so some extra structure (\emph{e.g.} the density property) is needed.  One might ask whether our main theorem holds with $X$ an Oka manifold instead.

\medskip

Thirdly, it was conjectured by Schoen and Yao \cite{SY} that no proper 
harmonic map could exist from the unit disk onto $\RR^2$.
The conjecture was recently disproved by Alarc\'{o}n and Galv\'{e}z \cite{AG}, 
but a much stronger result follows easily from our main theorem: 

\begin{theorem*}[\ref{thm:properharmonic}]
Let $\mathcal R$ be any open Riemann surface. Then $\mathcal R$
admits a proper harmonic mapping into $\RR^2$.
\end{theorem*}

\medskip

Finally, it is of general interest to find new methods to produce proper holomorphic 
maps from Riemann surfaces into complex manifolds, due to the long standing 
open problem whether any open Riemann surface admits a proper holomorphic embedding 
in $\CC^2$.

\section{Density Property}
\label{sec:density}

The density property was introduced in Complex Geometry by Varolin \cite{Varolin1}, \cite{Varolin2}. For a survey about the current state of research related to density property and Anders{\'e}n--Lempert theory, we refer to Kaliman and Kutzschebauch \cite{KalimanKutschebauch}.

\begin{definition}
\label{def:density}
A complex manifold $X$ has the \emph{density property} if in the compact-open
topology the Lie algebra $Lie_{hol}(X)$ generated by completely integrable holomorphic
vector fields on $X$ is dense in the Lie algebra $VF_{hol}(X)$ of all holomorphic vector
fields on $X$.
\end{definition}

\begin{definition}
\label{def:densityvol}
Let a complex manifold $X$ be equipped with a holomorphic volume
form $\omega$ (i.e. $\omega$ is nowhere vanishing section of the canonical bundle). We say that $X$
has the \emph{volume density property} with respect to $\omega$ if in the compact-open topology
the Lie algebra $Lie^\omega_{hol}(X)$ generated by completely integrable holomorphic vector fields $\nu$ 
such that $\nu(\omega) = 0$, is dense in the Lie algebra $VF^\omega_{hol}(X)$ of all holomorphic vector
fields that annihilate $\omega$.
\end{definition}

The following theorem is the central result of Anders{\'e}n--Lempert theory (originating from works of Anders\'en and Lempert \cite{Andersen}, \cite{AndersenLempert}), and  is given in the following form in \cite{KalimanKutschebauch} by Kaliman and Kutzschebauch, but essentially (for $\CC^n$) proven already in \cite{ForstnericRosay} by Forstneri\v{c} and Rosay.
\begin{theorem}
\label{thm:AndersenLempert}
Let $X$ be a Stein manifold with the density (resp. volume density) property and let $\Omega$ be an open subset of $X$. In case of volume density property further assume that $H^{n-1}(\Omega,\CC) = 0$. Suppose that $\Phi : [0, 1] \times \Omega \to X$ is a $\smoothone$-smooth map such that
\begin{enumerate}
\item $\Phi_t : \Omega \to X$ is holomorphic and injective (and resp. volume preserving) for every $t \in [0, 1]$
\item $\Phi_0 : \Omega \to X$ is the natural embedding of $\Omega$ into $X$
\item $\Phi_t(\Omega)$ is a Runge subset of $X$ for every $t \in [0, 1]$
\end{enumerate}
Then for each $\varepsilon > 0$ and every compact subset $K \subset \Omega$ there is a continuous family
$\alpha : [0, 1] \to \aut{X}$ of holomorphic (and resp. volume preserving) automorphisms of $X$ such that
\[
\alpha_0 = \id \; \mbox{and} \; \left| \alpha_t - \Phi_t \right|_K < \varepsilon
\]
for every $t \in [0, 1]$.
\end{theorem}

\begin{remark}
In the case of the volume density property it is enough to assume that $H^{n-1}(\Omega',\CC)=0$
for all connected components $\Omega'$ of $\Omega$ where $\Phi$ is not the identity map.  
The assumption is used to solve a certain differential equation which is trivially solvable on
the components where $\Phi$ is the identity.   
\end{remark}

Among one of the many results following from this theorem, we only need the following, first given by Varolin \cite{Varolin2}:
\begin{proposition}
\label{prop:transitive}
Let $X$ be a Stein manifold of dimension $n \geq 2$ with density (resp. volume density) property,
$K$ be a compact in $X$, and $x, y \in X$ be two points outside the convex hull of K.
Suppose that $x_1, \dots, x_m \in K$.
Then there exists a (resp. volume-preserving) holomorphic automorphism $\Psi$ of $X$
such that $\Psi(x_i) = x_i$ for every $i = 1, \dots, m$, $\Psi|K : K \to X$ is as close to the natural
embedding as we wish, and $\Psi(y) = x$.
\end{proposition}

\begin{corollary}
\label{cor:transitive}
Let $X$ be a Stein manifold of dimension $n \geq 2$ with density (resp. volume density) property, then its group of holomorphic automorphisms acts $m$-transitively for any $m \in \NN$.
\end{corollary}
\medskip
\par\noindent
Stein manifolds with (volume) density property are \emph{elliptic} in the sense of Gromov and satisfy the so-called Oka--Grauert--Gromov principle.

\begin{definition}
Let $X$ be a complex manifold. It is said to have the \emph{Convex Approximation Property}, if every holomorphic map of a compact convex set $K \subset \CC^n$ in $X$ can be approximated uniformly on $K$ by entire holomorphic maps $\CC^n \to X$.
\end{definition}
\noindent
In more recent terminology introduced by Forstneri\v{c} \cite{OkaManifolds}, a manifold satisfying the Convex Approximation Property is called an \emph{Oka manifold}. All elliptic Stein manifolds, in particular those with (volume) density property, are Oka manifolds.
\medskip
\par\noindent
Oka manifolds satisfy the following (see Drinovec-Drnov\v{s}ek and Forstneri\v{c} \cite{DrinovecForstneric3}):
\begin{theorem}
\label{thm:OkaPseudoConvex}
Let $S$ be a Stein manifold and $D \subset\subset S$ a strongly pseudoconvex domain with $\smoothany^\ell$ boundary ($\ell \geq 2$) whose closure $\overline{D}$ is $\mathcal{O}(S)$-convex, and let $Y$ be an Oka manifold. Let $r \in \{0, 1, \dots , \ell\}$ and let $f : S \to Y$ be a $C^r$-map which is holomorphic in $D$. Then $f$ can be approximated in the $C^r(\overline{D}, Y)$-topology by holomorphic maps $S \to Y$ which are homotopic to $f$.
\end{theorem}

\section{Approximating in the non-critical case}

The goal of this section is to prove the following approximation result.  

\begin{proposition}
\label{prop:glueannulus}
Let $\mathcal R_0 \subset\subset \mathcal R_1 \subset\subset \mathcal R_2$ be bordered Riemann surfaces, let $X$ be a Stein manifold with the density or volume density property, and let $K \subset X$ be a holomorphically 
convex compact set.  Let $f: \overline{\mathcal{R}_1} \to X$ be a holomorphic 
immersion, and assume that $f( \overline{ \mathcal R_1 \setminus \mathcal R_0} ) \subset X \setminus K$.  
Let $\Gamma$ be one of the boundary components of $\mathcal R_1$, and $\mathcal{A} \subset \mathcal{R}_2 \setminus \overline{\mathcal{R}_0}$ be an annulus containing $\Gamma$.

Then for any compact subset $L$ of $\mathcal{A}$ and any $\varepsilon > 0$ there exists 
a holomorphic immersion $g: \mathcal R_1 \cup L \rightarrow X$ such 
that the following holds:

\begin{enumerate}
 \item $\left\| g-f \right\|_{\mathcal R_0} < \varepsilon$, and
 \item $g( (\mathcal R_1 \cup L) \setminus \mathcal R_0) \subset X \setminus K$.
\end{enumerate}
\end{proposition}

This will in turn depend on the following result: 

\begin{proposition}
\label{prop:mainglue}
Let $\mathcal R_0 \subset\subset \mathcal R_1 \subset\subset \mathcal R_2$ be bordered Riemann surfaces, let $X$ be a Stein manifold with the density or volume density property, and let $K \subset X$ be a holomorphically 
convex compact set.  Let $f: \overline{\mathcal{R}_1} \to X$ be a holomorphic 
immersion, and assume that $f( \overline{ \mathcal R_1 \setminus \mathcal R_0} ) \subset X \setminus K$.  
Let $\Gamma$ be one of the boundary components of $\mathcal R_1$, let 
$V \subset \mathcal R_2 \setminus \mathcal R_0$ be an open set containing $\Gamma$, and assume
that $f|V$ is an embedding, with $f(\overline V)$ not intersecting $K\cup f(\overline{\mathcal R_0})$.

Let $U \subset \mathcal R_2 \setminus \mathcal R_0$ be a simply connected open set, $U \cap \mathcal R_1 \subset V \cap \mathcal R_1$, $U \cap \mathcal{R}_1$ connected,
and assume that we are given a $\smooth$-smooth isotopy $\varphi(\cdot,t) : U \to \mathcal R_2 \setminus \mathcal R_0$, $t \in [0,1]$, of injective holomorphic maps, such that the following holds:

\begin{enumerate}
 \item $\varphi(\cdot, 0) = \id_U$,
 \item $\varphi(U, 1) \subseteq V$,
 \item $\varphi(U \setminus\mathcal R_1, t) \subseteq \mathcal R_2 \setminus \mathcal R_1$ for all $t \in [0, 1]$, and
 \item $\varphi(U \cap \mathcal R_1, t) \subseteq V \cap \mathcal R_1$ for all $t \in [0, 1]$.
\end{enumerate}

Then for any compact subset $L$ of $U$ and any $\varepsilon > 0$ there exists 
a holomorphic immersion $g: \mathcal R_1 \cup L \rightarrow X$ such 
that the following holds:

\begin{enumerate}
 \item $\left\| g-f \right\|_{\mathcal R_0} < \varepsilon$, and
 \item $g( (\mathcal R_1 \cup L) \setminus \mathcal R_0) \subset X \setminus K$.
\end{enumerate}

\begin{figure}[H]
\includegraphics[width=0.6\textwidth]{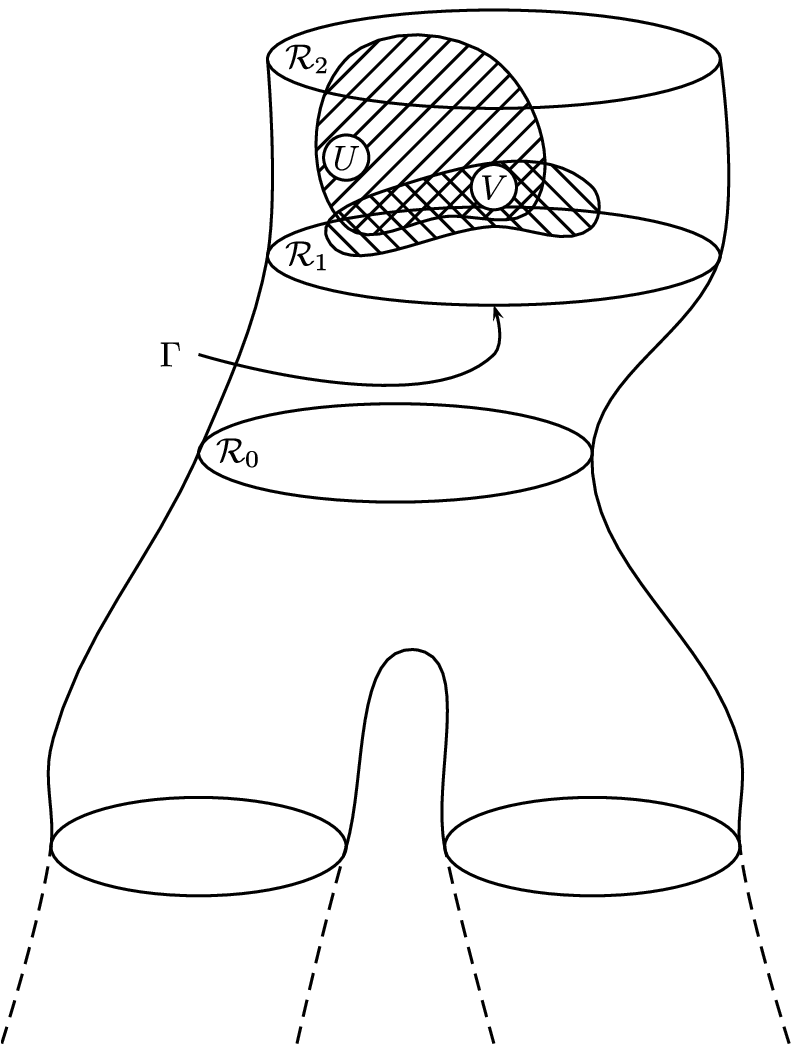}
\caption{\label{fig:mainglue}}
\end{figure}

\end{proposition}

We cite Theorem 4.1 from Forstneri\v{c} \cite{Forstneric} which will be needed in the proof:
\begin{theorem}
\label{thm:compositionsplit}
Let $A$ and $B$ be compact sets in a complex manifold $X$ such that $D = A \cup B$ has a basis of Stein neighbourhoods in $X$ and $\overline{A \setminus B} \cap \overline{B \setminus A} = \emptyset$. Given an open set $\tilde{C} \subseteq X$ containing $C := A \cap B$ there exist open sets $A' \supseteq A, B' \supseteq B, C' \supseteq C$ with $C' \subseteq A' \cap B' \subseteq \tilde{C}$, satisfying the following: For every injective holomorphic map $\gamma : \tilde{C} \to X$ which is sufficiently uniformly close to the identity on $\tilde{C}$ there exist injective holomorphic maps $\alpha : A' \to X, \beta : B' \to X$, uniformly close to the identity on their respective domains and satisfying
\[ \gamma = \beta \circ \alpha^{-1}\]
\end{theorem}
\par\noindent
In analogy to the classical splitting for Cartan pairs, we call such $A$ and $B$ satisfying the assertions of the theorem also a \emph{Cartan pair}.
\medskip

We will also need the following: 

\begin{proposition}[Hilfssatz 11 in \cite{ForsterRamspott-Endromis}]
Let X be a Stein manifold and Y be an analytic submanifold. Then there exists a biholomorphic map of an open neighbourhood of Y in X onto an open neighbourhood of the zero section of the normal bundle of Y in X, mapping Y biholomorphically to the zero section.
\end{proposition}

\begin{corollary}\label{normalbundle2}
Let $X$ be a complex manifold of dimension $n$ and let $f:\overline{\mathcal R}\rightarrow X$ be an 
immersion of a bordered Riemann surface $\mathcal R$.  Then there exists a 
holomorphic immersion $F:\overline{\mathcal R}\times\udisk^{n-1}\rightarrow X$
such that $F|_{\mathcal R\times\{0\}}=f$.
\end{corollary}
\begin{proof}
Let $\tilde X$ be a complex manifold with an embedding $\tilde f:\overline{\mathcal R}\rightarrow\tilde X$
and an immersion $\rho:\tilde X\rightarrow X$ such that $f=\rho\circ\tilde f$.  By Siu's theorem 
\cite{siu} we may assume that $\tilde X$ is Stein, and so the proposition applies to give 
an embedding $\tilde F$ into $\tilde X$.  Note that any vector bundle over an open Riemann surface 
is trivial, and define $F=\rho\circ\tilde F$.  
\end{proof}

\begin{lemma}
\label{lem:runge}
Let $X$ be a Stein manifold, let $K\subset X$ be 
a holomorphically convex compact set, and let $f:\overline{\udisk} \to X\setminus K$ be an embedding.  
Let $V'\subset X\setminus K$ be an open neighbourhood of $f( \overline{\udisk} )$ and 
assume given an isotopy of holomorphic injections $\phi_t : V' \to X\setminus K$. Then there
exists an open neighbourhood $V'' \subseteq V'$ of $f( \overline{\udisk} )$ and 
an open neighbourhood $W$ of $K$
such that 
$\Omega_t = \phi_t(V''\cup W)$ is a Runge domain in $X$ for all $t \in [0,1]$. 
\end{lemma}

\begin{proof}

It is a well known fact that $K\cup\phi_t(f(\overline{\udisk}))$ is 
holomorphically convex for each fixed $t\in [0,1]$ (for lack of a reference we 
include an argument below).
The result is then a consequence of 
of Lemma 2.2 in \cite{ForstnericRosay} formulated for $X$ instead of $\CC^n$
(the proof in the case of a Stein manifold is identical).  \

We now show that $K\cup\phi_t(f(\overline{\udisk}))$ is 
holomorphically convex for each fixed $t\in [0,1]$.  Let 
$r>1$ be chosen close enough to $1$ such that 
$(\phi_t\circ f):\overline{\udisk}_r\to X\setminus K$ is an embedding, 
let $\Sigma=\phi_t(f(\boundary\udisk))$, and $\Sigma'=\phi_t(f(\boundary\udisk_r))$.
We want to show that $\widehat{K\cup\Sigma}=K\cup\phi_t(f(\overline{\udisk}))$. \

By Theorem 12.5. in \cite{AlexanderWermer} and the fact that $X$ embeds
properly in $\CC^N$ for $N$ sufficiently large, we have that 
$\widehat{K\cup\Sigma'}\setminus (K\cup\Sigma')$ (resp. $\widehat{K\cup\Sigma}\setminus (K\cup\Sigma)$)
is a one-dimensional analytic subset of $X\setminus (K\cup\Sigma')$ (resp. $X\setminus (K\cup\Sigma$).
Note first that $\widehat{K\cup\Sigma}$ cannot contain a relatively open 
subset of $\phi_t(f(\udisk_r\setminus\overline{\udisk}))$.  If it did, it would, 
by the identity principle for analytic sets, contain $\phi_t(f(\udisk_r\setminus\overline{\udisk}))$, 
and so $\widehat{K\cup\Sigma}\setminus K$ would be an analytic subset of $X\setminus K$.   This
is impossible since $K$ is holomorphically convex.
Since 
$\widehat{K\cup \phi_t(f(\overline{\udisk}))}\subset\widehat{K\cup\Sigma'}$ we get that 
\[
\overline{
\widehat{K\cup\phi_t(f(\overline{\udisk}))}\setminus (K\cup\phi_t(f(\overline{\udisk})))
}
\cap (K\cup\phi_t(f(\overline{\udisk}))) = K\cup A, 
\]
where $A$ is a finite set of points.   By Rossi's local maximum principle we have
that $\widehat{K\cup\phi_t(f(\overline{\udisk}))}=(K\cup\phi_t(f(\overline{\udisk})))\cup\widehat{K\cup A}$ which 
implies that $K\cup \phi_t(f(\overline{\udisk}))$ is holomorphically convex.

\end{proof}

\noindent

\begin{proof}[Proof of Proposition \ref{prop:mainglue}] \hfill \\
Since $X$ is an Oka manifold we may assume, by approximation, that $f$ is already 
defined on $\overline{\mathcal R_2}$; the task is is to find an approximation which achieves (2). Note that $K\cup f(\overline{\mathcal R_0})$ is holomorphically convex.  

\medskip

Define $A:=\overline{\mathcal R_1}$, and let $B\subset\mathcal R_2$ be a Stein compact  
such that the pair A, B is a Cartan pair as in Theorem \ref{thm:compositionsplit}, $A\cap B$ simply connected and contained in $V$, and $L\subset (A\cup B)^\circ$.  We will approximate $f$ on a certain thickening of $A\cup B$ in $\mathcal R_2\times\CC^{n-1}$ which will allow us to exploit the density property of $X$. 

\medskip

Since $f$ is an immersion we have by Corollary \ref{normalbundle2} that $f$ extends to an immersion 
\[
F:\mathcal R_2\times\mu\cdot\udisk^{n-1}\rightarrow X, 
\]
such that $F|_{\mathcal R_2\times\{0\}}=f$.  We may assume that $F|_{\overline V\times\mu\cdot\udisk^{n-1}}$ is an embedding whose image does not intersect $K\cup f(\overline{\mathcal R_0})$. \medskip

Set $\tilde\omega:=F^*{\omega}$.  By choosing $\mu_1$ small enough we have that 
$\varphi_t$ extends to an isotopy $\phi_t:U\times\mu_1\cdot\udisk^{n-1} \to (\mathcal R_2\setminus\mathcal R_0)\times\mu\cdot\udisk^{n-1}$ of the form 
\[
\phi_t(x,w)=(\varphi_t(x),\sigma_t(x,w)), \sigma_t(x,0)=0,
\]
and such that $\phi_t^*\tilde\omega=\tilde\omega$ for all $t\in [0,1]$.
For $\mu_1 > 0$ small enough $\sigma_t$ can be found easily in such local coordinates where $\omega$ is the standard volume form.

Now the following is our strategy: note that there is an open neighbourhood $W$ of 
$C=A\cap B$, relatively compact in $V$, such that 
on the image $\Omega:=F(W\times\mu_1\cdot\udisk^{n-1})$ we have a well 
defined isotopy 
\[
\Phi_t:=F\circ\phi_t\circ F^{-1}:\Omega\rightarrow X\setminus (K\cup f(\overline{\mathcal R_0})).
\]
Choose $W$ such that $H^{n-1}(W\times\mu_1\cdot\udisk^{n-1},\CC)=0$, and
note that $\Phi_t^*\omega=\omega$ for all $t$.
Note also that the composition $F_B:=F\circ\phi_1$ is well defined near $B\times\mu_1\cdot\udisk^{n-1}$.  We will approximate $\Phi_1$ well enough by an automorphism $\Lambda$ of $X$, essentially fixing $K\cup f(\overline{\mathcal R_0})$, 
such that the map $\Lambda\circ F$ may be glued with minor perturbations to the map $F_B$.

\medskip

By Lemma \ref{lem:runge} there exists a neighbourhood $\Omega_1$ of $K\cup f(\overline{\mathcal R_0})$ and a neighbourhood 
$\Omega_2\subset\Omega$ of $f(\overline W)$ such that $\Phi_t(\Omega_1\cup\Omega_2)$
is Runge for each $t$, where $\Phi_t|_{\Omega_1}\equiv \id$.
Fix a $0<\mu_2<\mu_1$ such that $F(\overline W\times\mu_2\cdot\udisk^{n-1})\subset\Omega_2$.  For any $0<\delta<\mu_2$  let $A_\delta,B_\delta$ denote the Cartan pairs $A\times\delta\cdot\overline{\udisk^{n-1}}$ and $B\times\delta\cdot\overline{\udisk^{n-1}}$ respectively.  Let $\tilde C_\delta$
be a neighbourhood of $C_\delta:=A_\delta\cap B_\delta$ contained in $W\times\mu_2\cdot\udisk^{n-1}$, and let $C'_\delta$ be the corresponding neighbourhood of $C_\delta$
in Theorem \ref{thm:compositionsplit}.

\medskip

Now let $\Lambda_j$ be a sequence of automorphisms of $X$ converging uniformly 
to $\Phi_1$ near $F(C_\delta')$ and such that each $\Lambda_j$ stays 
uniformly close to the identity near $K\cup f(\overline{\mathcal R_0})$. 
This is possible by the (volume) density 
property (Theorem \ref{thm:AndersenLempert} and the remark following it) of $X$ and the choices made above. 

Then 
$
\gamma_j:=F^{-1}\circ\Lambda_j^{-1}\circ F\circ\phi_1
$
converges to the identity uniformly on $C_\delta'$.  Decompose $\gamma_j=\alpha_j\circ\beta_j^{-1}$
using Theorem \ref{thm:compositionsplit}.  Define $G_j:=\Lambda_j\circ F\circ\alpha_j$ on $A_\delta'$
and $G_j:=F\circ\phi_1\circ\beta_j$ on $B'_\delta$.  Now put $G:=G_j$
for a large enough $j$ and then $g:=G|_{(A\cup B)\times\{0\}}$.  
\end{proof}

\begin{proof}[Proof of Proposition \ref{prop:glueannulus}] \hfill
\begin{enumerate}
\item The extension from $f$ to $g$ will be achieved in two steps, attaching in each step a simply connected domain in $\mathcal{R}_2$ using Proposition \ref{prop:mainglue}.
Since $f$ is defined on $\overline{\mathcal{R}_1} \subset \mathcal{R}_2$, it extends immersively to a neighbourhood $\mathcal{R}^\prime_1 \subset\subset \mathcal{R}_2$ such that $f$ is injective on a neighbourhood $V \subset \mathcal{R}_2$ of $\boundary \mathcal{R}^\prime_1$ which is generically the case. 
\begin{figure}[H]
\begin{centering}
\subfigure[Overlapping (shaded) of $L_1$ and $L_2$ \label{fig:glueannulus-decomposition}]{
\includegraphics[width=0.4\textwidth]{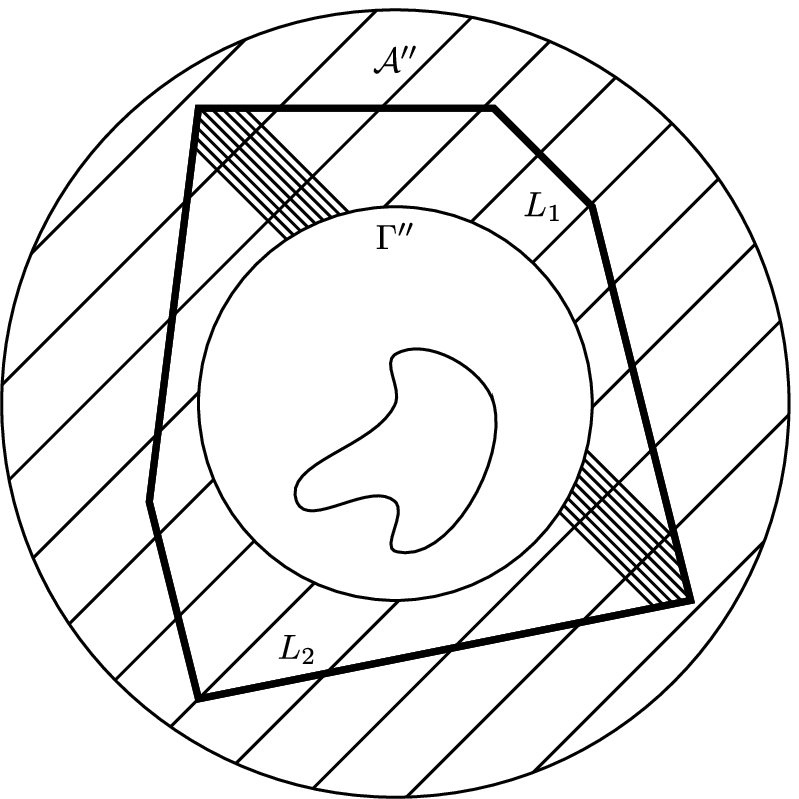}
}\subfigure[Glueing the disk \label{fig:glueannulus-shrinking}]{
\includegraphics[height=0.4\textwidth]{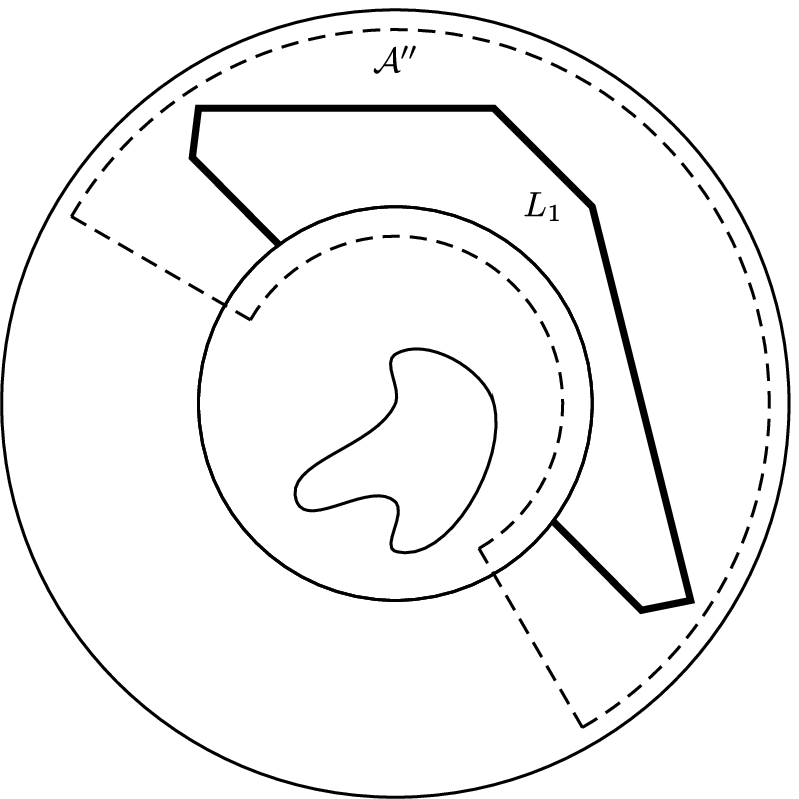}
}
\end{centering}
\caption{Glueing an annulus to a bordered Riemann surface\label{fig:glueannulus}}
\end{figure}

\item We embed the annulus $\mathcal{A}$ in $\CC$ as a planar domain and arrange it by an uniformizing map such that $\mathcal{A} \setminus \mathcal{R}_1^\prime \to \mathcal{A}^\prime$. Now we can work entirely in $\CC$ and identify all subsets of $\mathcal{A}$ with subsets of $\CC$ in order to give the sets $U$ and isotopies $\varphi$ needed for Proposition \ref{prop:mainglue}. The curve $\Gamma$ is mapped to a smooth curve $\Gamma^\prime$ in the the image $\mathcal{A}^\prime$ of $\mathcal{A}$. 
Let $D$ denote the bounded domain in $\CC$ bordered by $\Gamma^\prime$. Then $D$ is homeomorphic to a disk and  $\mathcal{A}^\prime \setminus D$
is again an annular region, which we can identify with an annulus $\mathcal{A}^{\prime\prime} =: \ell \udisk \setminus \overline{\udisk}, \, \ell > 0$ via another
uniformizing map. This map is a $\smooth$-smooth diffeomorphism up to the boundary and therefore extends to a small neighbourhood.
The set $L \cap \mathcal{A}^{\prime\prime}$ can be written as a union of two compact overlapping sets $L_1$ and $L_2$ which are both homotopic to disks inside in $\mathcal{A}^{\prime\prime}$, as depicted in figure \ref{fig:glueannulus-decomposition}.

Set $f_0 := f$. The immersion after the first extension to any compact $L_1 \subset U_1 \subset \mathcal{A}'$ will be denoted by $f_1$, and after the second extension to $L_2 \subset U_2 \subset \mathcal{A}'$ by $f_2 = g$.

\item Define $U_1$ to be
\[
U_1 := \left\{ r \cdot e^{i \theta} \in \CC \, : \, (1 - \delta) < r < \ell(1 - \delta), \alpha < \theta < \beta \right\}
\]
where $1 > \delta > 0$, and $\alpha, \beta \in \RR$ are such that $L \subset (1 - \delta) \ell \udisk$.
The $\smooth$-smooth isotopy $\varphi_1(z, t), \; t \in [0,1],$ of holomorphic injections is given explicitly in equation \eqref{eq:firstisotopy} below:
\begin{equation}
\label{eq:firstisotopy}
\varphi_1(z, t) = \exp\left(\log(z) \cdot \left( (\gamma - 1) \cdot t  + 1 \right) \right)
\end{equation}
with $\log$ defined on $\CC \setminus \RR^-$ and suitable angle $\gamma \in \RR$.
We apply Proposition \ref{prop:mainglue} to extend $f_0 : \mathcal{R}^\prime_1 \subset\subset \mathcal{R}_2 \to X$ approximately up to $\varepsilon/2$ to $L_1 \subset U_1$ using the isotopy $\varphi_1$ and denote the approximation by $f_1$.
\item Now consider $\mathcal{A}^{\prime\prime} \setminus L_1$ which is again homeomorphic to an annulus and can be mapped to $\mathcal{A}^{\prime\prime\prime} =: \ell^\prime \udisk \setminus \overline{\udisk}, \, \ell^\prime > 0$ by another uniformizing map. Then we are back in the previous situation and define $U_2$ and $\varphi_2$ the same way. This leads to the desired approximation $g = f_2$. \qedhere
\end{enumerate}
\end{proof}

\section{Approximating in the critical case}

\begin{definition}
\label{def:nicemorse}
A strongly subharmonic exhaustion function $\rho$ of a Riemann surface $\mathcal R$ is said to be a Morse function with \emph{nice singularities}
if any critical point $\xi$ is either a local minimum, or there exist local coordinates 
$z = x+iy : U_\xi \to \CC$ such that $\rho$ is of the form 
\[
\rho(z)= \rho(\xi) + x^2 - \mu \cdot y^2, 
\]
for some $\mu \in (0,1]$.
\end{definition}

\begin{figure}[H]
\begin{centering}
\subfigure[1\textsuperscript{st} possibility]{
\includegraphics[width=0.33\textwidth]{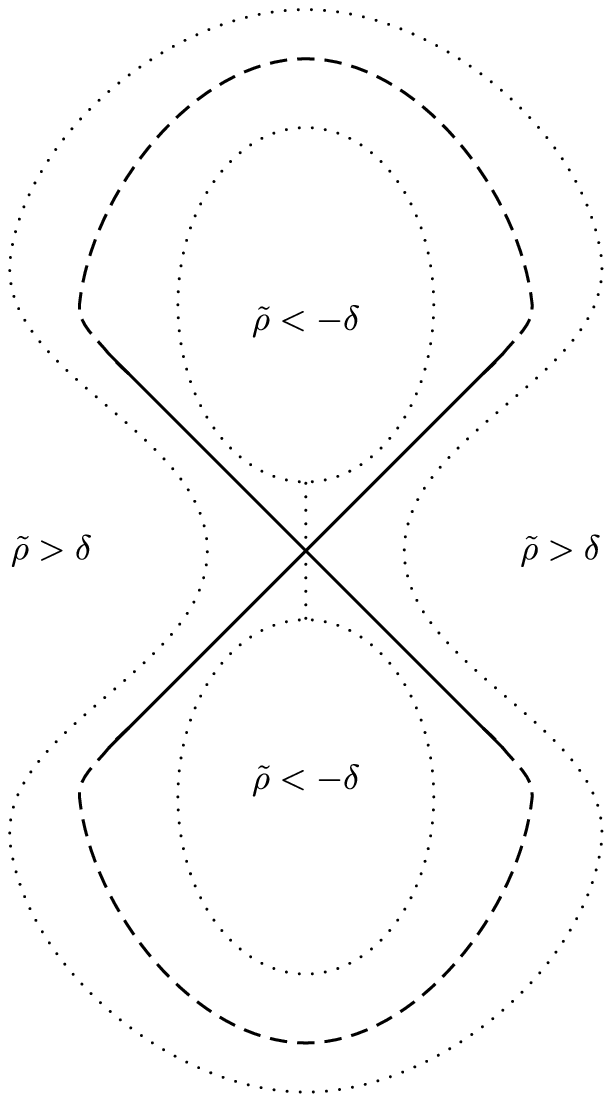}
}\subfigure[2\textsuperscript{nd} possibility]{
\includegraphics[height=0.33\textwidth]{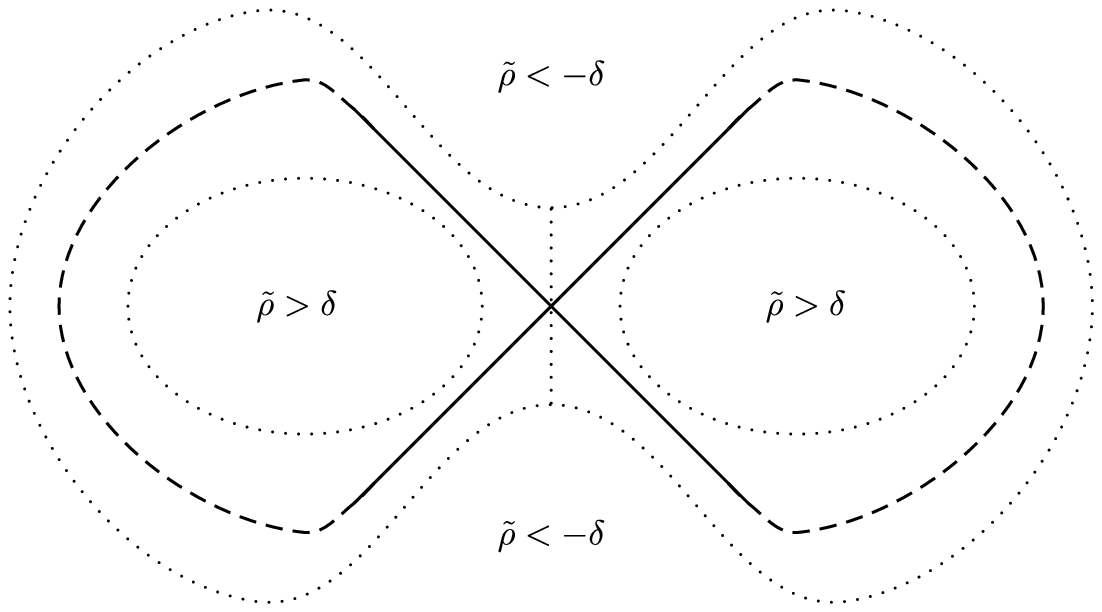}
}
\end{centering}
\caption{Critical points for the strongly plurisubharmonic exhaustion function $\rho : \mathcal{R} \to \RR$ at $\rho = 0$ and the level sets of $\rho$
\label{fig:morsecritical}}
\end{figure}

\begin{proposition}
\label{prop:gluecritical}
Let $\mathcal R$ be an open connected Riemann surface, and let $\rho\in\mathcal C^\infty(\mathcal R)$
be a Morse exhaustion function with nice singularities. Let $\xi\in\mathcal R$
be a critical point of $\rho$ which is not a local minimum, and let $c=\rho(\xi)$.
Then there exists a $\delta > 0$ such that the following holds: 
Let $X$ be a Stein manifold, 
let $K\subset X$ be a holomorphically 
convex compact set with $X\setminus K$ connected, let $f:\mathcal R_{c-\delta}\rightarrow X$ be a holomorphic immersion 
with $f(\boundary\mathcal R_{c-\delta})\in X\setminus K$,  and let $\epsilon>0$.   Then 
there exists a holomorphic immersion $\tilde f:\mathcal R_{c+\delta}\rightarrow X$
such that 
\begin{enumerate}
\item $\| \tilde f-f \|_{\mathcal R_{c-\delta}}<\epsilon$
\item $\tilde f(\mathcal R_{c+\delta}\setminus\mathcal R_{c-\delta})\subset X\setminus K$.
\end{enumerate}
\end{proposition}

\begin{proof}
We will describe first how to cross the connected component of $\{\rho = c\}$ which 
contains the critical point $\xi$; crossing the other components are 
done by applying Proposition \ref{prop:glueannulus}.
Let $z: U_\xi \to \CC$ be local coordinates with $z(\xi)=0$, and 
such that, in local coordinates, $\rho(x,y)=x^2-\mu\cdot y^2$, $0<\mu\leq 1$ (for simplicity we assume $\rho(\xi)=0$).  To 
get a clear picture of what is going on we start by embedding a neighbourhood of the connected 
component $\Gamma$ of $\{\rho=0\}$ that intersects $U_\xi$ into $\CC$.  In 
local coordinates $\Gamma_\xi=\Gamma\cap U_\xi$ is the union 
of the straight lines $\gamma_{\pm}=\{x = \pm\sqrt\mu\cdot y\}$.
Let $\tilde\gamma_{\pm}$ denote the preimages if these lines in $U_\xi$.
Since $\{\rho=0\}$ is smooth outside $\xi$ we have that $\Gamma$ is obtained by 
attaching two smooth arcs $l_j, j = 1, 2,$ to the endpoints of the arcs $\tilde\gamma_{\pm}$.
This means that $\Gamma$ may be written as the union of two closed curves $\Gamma_1$ and $\Gamma_2$
intersecting at a single point $\xi$.  Assume that $\Gamma_1$ is obtained 
by attaching an end-point of $l_1$ to the end-point of $\tilde\gamma_+$ which 
in local coordinates lies in the upper half plane $H^+$. Then the other end-point of $l_1$ must
be attached to one of the end-points of $\tilde\gamma_-$. Otherwise, since 
$\mathcal R$ may be oriented, there would exists a path starting in $\{\rho<0\}$, 
never crossing $\{\rho=0\}$, and ending up in $\{\rho>0\}$. This means 
that the coordinate function $z$ may be extended smoothly to $\Gamma_1\setminus U_\xi$
such that the image $z(\Gamma_1\setminus\xi)$ is either completely 
contained in $H^+$ or in the right half-plane $H^R$.  Likewise, 
the coordinate function $z$ may be extended smoothly to $\Gamma_2\setminus U_\xi$
such that the image $z(\Gamma_2\setminus\xi)$ is either completely 
contained in $H^-$ or in $H^L$.  Approximating the extended map $z$ using  
Mergelyan's Theorem, we get an embedding $\tilde z:\Omega\rightarrow\CC$
of an open neighbourhood $\Omega$ of $\Gamma$, such that 
the image together with the level sets of the strictly subharmonic function $\tilde\rho=\rho\circ\tilde z^{-1}$ is described by Figure \ref{fig:morsecritical}.  \

Let $\tilde\Omega=\tilde z(\Omega), \tilde\Gamma=\tilde z(\Gamma)$.  By Figure \ref{fig:morsecritical} it is clear that in 
both cases we may assume that $\tilde\Omega$ (topologically)
is a disk with two holes taken out.  \

We now consider case (a).  Choosing $\delta>0$ 
small enough it is clear that $\{\tilde\rho=-\delta\}$
is the union of two smooth closed curves, one in 
each bounded component of $\CC\setminus\tilde\Gamma$, 
and $\{\tilde\rho=\delta\}$ is a single smooth closed curve 
in the unbounded component of $\CC\setminus\tilde\Gamma$.
Fix a $\delta$ small enough, and let $\tilde\sigma$ denote
the vertical straight line segment passing through 
the origin and connecting the two components of $\{\tilde\rho<-\delta\}$.
Define $\sigma=\tilde z^{-1}(\tilde\sigma)$.   It clear from 
Figure \ref{fig:morsecritical}. that $\{\tilde\rho<\delta\}\setminus (\{\tilde\rho\leq-\delta\}\cup\tilde\sigma)$
has the topological type of an annulus.  Therefore, by Proposition \ref{prop:glueannulus}, 
it is enough to approximate $f$ by a map $g$ which is 
holomorphic on a neighbourhood of $C=\{\rho\leq-\delta\}\cup\sigma$, 
with $g(\sigma)\cap K=\emptyset$. \

Note first that $C$ is holomorphically convex in $\mathcal R$:
if $\zeta\in\mathcal R\setminus C$ we want to find
an continuous path in $\mathcal R\setminus C$ between $\zeta$ and  a point in $\{\rho>0\}$.
Clearly there is a path in $\mathcal R\setminus\{\rho\leq -\delta\}$ connecting 
$\zeta$ and some point in $\{\rho>0\}$.  If this path does not intersect $\sigma$
we are done. If it intersects $\sigma$, it is clear from Figure \ref{fig:morsecritical} that we 
may modify the path so that it does not.  \

Now since $X\setminus K$ is connected we may find a smooth extension $\tilde f$
of $f$ to $\sigma$ such that $\tilde f=f$ on some neighbourhood of $\{\rho\leq -\delta\}$
and such that $\tilde f(\sigma)\in X\setminus K$.  By Mergelyan's Theorem 
and the fact that $X$ is Stein, 
we have that $\tilde f$ may be approximated arbitrarily well by holomorphic maps.  \

The case (b) is now dealt with in a similar manner, note only that in this case, 
$\{\tilde\rho<\delta\}\setminus C$ is the union of two disjoint 
annuli, hence Proposition \ref{prop:glueannulus} must be applied twice after extending (approximating)
the map to $\sigma$. \

Finally we also need to approximate when crossing the 
remaining components of $\{\rho=0\}$.    If $\delta$ is chosen 
small enough we have that there are no critical points other than $\xi$
in $\{-\delta<\rho<\delta\}$, and so this approximation 
is furnished by Proposition \ref{prop:glueannulus}.
\end{proof}

\section{Main theorem and Applications}

\begin{theorem}
\label{thm:embedsurface}
Let $X$ be a Stein manifold with density property or with volume density property and $\mathcal{R}$ an open Riemann surface.
If $\mathcal{R} \not \equiv \CC$, then assume further there exists a strongly plurisubharmonic exhaustion function $\tau$ of $X$ with increasing compact sublevel sets $K_j = \tau^{-1}([-\infty, M_j])$, $M_{j+1} > M_j > 0$, $j \in \NN$, $\displaystyle \lim_{j \to \infty} M_j = \infty$, such that $X \setminus K_j$ is connected for all $j \in \NN$.
\begin{itemize}
 \item[(a)] If $\dim X \geq 3$ then there is a proper holomorphic embedding $\mathcal{R} \hookrightarrow X$.
 \item[(b)] If $\dim X = 2$ then there is a proper holomorphic immersion $\mathcal{R} \to X$.
\end{itemize}
\end{theorem}

The main ingredients for the proof are the Propositions \ref{prop:glueannulus} and \ref{prop:gluecritical} from the previous sections. They will be used in an inductive framework provided by Lemma 6.3 from Drinovec-Drnov\v{s}ek and Forstneri\v{c} \cite{DrinovecForstneric} which we cite here:
\begin{lemma}
\label{lem:shiftaway}
Let $X$ be an irreducible complex space of dimension $n \geq 2$, and
let $\tau: X \to \RR$ be a smooth exhaustion function which is $(n - 1)$-convex on
$\{x \in X: \tau(x) > M_1\}$. Let $\mathcal R$ be a finite Riemann surface, let $P$ be an open
set in $\CC^N$ containing $0$, and let $M_2 > M_1$. Assume that $f :  \overline{\mathcal R} \times P \to X$ is a
spray of maps of class $\mathcal{A}^2(\mathcal R)$ with the exceptional set $\sigma \subset \mathcal R$ of order $k \in \NN$,
and $U \subset \mathcal R$ is an open subset such that $f_0(z) \in \{x \in X_{\mbox{reg}} : \tau(x) \in (M_1, M_2)\}$ for all $z \in \overline{\mathcal R} \setminus U$.
Given $\varepsilon > 0$ and a number $M_3 > M_2$, there exist a domain
$P' \subset P$ containing $0 \in \CC^N$ and a spray of maps $g : \overline{\mathcal R} \times P' \to X$ of class
$\mathcal{A}^2(\mathcal R)$, with exceptional set $\sigma$ of order $k$, satisfying the following properties:
\begin{enumerate}
 \item $g_0(z) \in \{x \in X_{\mbox{reg}} : \tau(x) \in (M_2, M_3)\}$ for $z \in \boundary \mathcal R$,
 \item $g_0(z) \in \{x \in X: \tau(x) > M_1\}$ for $z \in \overline{\mathcal R} \setminus U$,
 \item $d(g_0(z), f_0(z)) < \varepsilon$ for $z \in U$, and
 \item $f_0$ and $g_0$ have the same $k$-jets at each of the points in $\sigma$.
\end{enumerate}
Moreover, $g_0$ can be chosen homotopic to $f_0$.
\end{lemma}

First we note, that $X$ in our case will be a Stein manifold and therefore $\tau$ can be taken to be a strongly plurisubharmonic exhaustion function, and we have $X = X_{\mbox{reg}}$ as well. The existence of a metric follows in the general case from para-compactness, but in our case of Stein manifolds we can work with the restriction of an euclidean norm.
We also cite from \cite{DrinovecForstneric2} their definition of spray of maps of class $\mathcal{A}^2(\mathcal R)$:

\begin{definition}
Assume that $X$ is a complex manifold, $\mathcal{R}$ is a relatively compact strongly pseudoconvex
domain with $\smoothany^2$ boundary in a Stein manifold $S$, and $\sigma$ is a finite set of points in
$\mathcal{R}$. A spray of maps of class $\mathcal{A}^2(\mathcal{R})$ with the exceptional set $\sigma$ of order $k \in \NN$ (and with
values in $X$) is a map $f : \overline{\mathcal{R}} \times P \to X$, where $P$ (the parameter set of the spray)
is an open subset of a Euclidean space $\CC^m$ containing the origin, such that the
following hold:
\begin{enumerate}
 \item $f$ is holomorphic on $\mathcal{R} \times P$ and of class $\smoothany^2$ on $\overline{\mathcal{R}} \times P$
 \item the maps $f(\cdot, 0)$ and $f(\cdot, t)$ agree on $\sigma$ up to order $k$ for $t \in P$, and
 \item for every $z \in \overline{\mathcal{R}} \setminus \sigma$ and $t \in P$ the map
\[ \partial_t f(z, t) : T_t\CC^m = \CC^m \to T_{f(z,t)}X \]
is surjective (the domination property).
\end{enumerate}
The map $f_0 = f(\cdot, 0)$ is called the core (or central) map of the spray $f$.
\end{definition}

In our case, $S$ will be a fixed finite open Riemann surface and in any step $\mathcal R$ the sublevel set of a strongly plurisubharmonic exhaustion function of $S$.
The core map of the spray will be the holomorphic immersion $f$ of $\overline{\mathcal R}$ into the complex manifold $X$ of dimension $n$, and 
we construct a spray as follows:  Let $\mathcal R'$ be a 
Riemann surface with $\mathcal R\subset\subset\mathcal R'$.  Since the tangent bundle of $\overline{\mathcal R'}$ is 
trivial, we may choose a non-vanishing holomorphic vector field $V$ on 
$\overline{\mathcal R'}$, and we let $\varphi_t$ denote its flow.  Define a
map $\tilde f : \overline{\mathcal R} \times \udisk^{n-1}\times \delta\cdot\udisk \to \mathcal R'\times\udisk^{n-1}$ by 
$(z,t_1,t_2)\mapsto (\varphi_{t_2}(z),t_1)$.  Choose an 
immersion $F:\overline{\mathcal R'}\times\udisk^{n-1}\rightarrow X$ according to Corollary \ref{normalbundle2}, and define $f:=F\circ\tilde f$.

\begin{proof}[Proof of Theorem \ref{thm:embedsurface}] \hfill
\begin{itemize}

\item[1.]
Let $\mathcal R$ be an open connected Riemann surface.  Since $\mathcal R$ is a 
Stein manifold we have that $\mathcal R$ admits a $\mathcal C^2$ strictly subharmonic exhaustion 
function $\rho:\mathcal R \to \RR^+$. Since strict subharmonicity is 
stable under small $\smoothtwo$-perturbations we may assume that $\rho$ is 
a Morse function, meaning that all critical points of $\rho$ are non-degenerate, 
and if $\xi$ and $\xi^\prime$ are two critical points of $\rho$ then $\rho(\xi) \neq \rho(\xi^\prime)$.
By Lemma 2.5 in \cite{HenkinLeiterer} we may further assume that any 
critical point $\xi$ is either a local minimum, or there exist local coordinates 
$z = x+iy : U_\xi \to \CC$ such that $\rho$ is of the form 
\[
\rho(z)= \rho(\xi) + x^2 - \mu \cdot y^2, 
\]
for some $\mu \in (0,1)$, i.e. that $\rho$ has only nice singularities.
In the following we denote by $\{\xi_k\}_{k \in I \subseteq \NN}$ the critical points of $\rho$, and by $c_k := \rho(\xi_k)$ the corresponding critical values, where $I$ is either $\NN$ or a $I = [1, \dots k_{\max}]$ for some $k_{\max} \in \NN$.  If there is only a 
finite number of critical points, we define inductively $c_{k+1}:=c_k+1$ for 
$k\geq k_{\max}$.

Let $\tau$ denote a strictly plurisubharmonic exhaustion function of the Stein manifold $X$. Choose a sequence of real $\varepsilon_k > 0$ such that $\sum_{k=1}^\infty \varepsilon_k < 1$. By
\[
 \mathcal{R}_\gamma := \{ z \in \mathcal{R} \, : \, \rho(z) < \gamma \}, \quad \gamma \in \RR
\]
 we denote the $\gamma$-sublevel set of $\rho$.
 
 \medskip

\item[2.]
For each $k\geq 2$ we do the following: if $\xi_k$ is a local minimum we put
$\delta_k:=\frac{1}{2}\mathrm{min}\{c_k-c_{k-1},c_{k+1}-c_k\}$, 
and otherwise we choose a small $\delta_k$ according to Proposition \ref{prop:gluecritical}. 
such that $\overline{\mathcal R_{c_k+\delta_k}}\setminus\mathcal R_{c_k-\delta_k}$
contains no other critical point of $\rho$ than $\xi_k$.   In the case of finitely many critical points
we put $\delta_k=0$ for $k>k_{\max}$.

\medskip

\item[3.]
Choose an initial embedding $f_2:\overline{\mathcal R_{c_2-\delta_2}}\rightarrow X$.
This is trivial since $\mathcal R_{c_2-\delta_2}$ is a disk.  We will now describe 
an inductive procedure how to construct immersions (resp. embeddings) $f_k$
of $\overline{\mathcal R_{c_k-\delta_k}}$ into $X$.  Assume that we have constructed 
immersions (resp. embeddings) $f_k:\overline{\mathcal R_{c_k-\delta_k}}\rightarrow X$
and real numbers  $r_k$ for $k=2,...,N$, $r_k\geq r_{k-1}+1$, and assume 
that $f_k(\overline{\mathcal R_{c_k-\delta_k}}\setminus\mathcal R_{c_{k-1}})\subset X\setminus K_{r_k}$, $\|f_{k}-f_{k-1}\|_{\overline{\mathcal R_{c_{k-1}-\delta_{k-1}}}}<\epsilon_{k-1}$.
We will now describe the inductive step how to construct $f_{N+1}$.

\begin{enumerate}
\item[(a)]\label{mainthmproof-away}  Choose $r_{N+1}\geq r_N+1$ such that $X\setminus K_{r_{N+1}}$ is connected.
We may assume that $f_N(\boundary\mathcal R_{c_N-\delta_N})\subset X\setminus K_{r_{N+1}}$:
since $f_N$ lives on a neighborhood of $\overline{\mathcal R_{c_N-\delta_N}}$, 
we may thicken $f_N$ as described above, and the boundary may be pushed 
away using Lemma \ref{lem:shiftaway}. 
\item[(b)] In the case of finitely many critical points, if $N>k_{\max}$, we may reach the next level set by attaching finitely many annuli to $\mathcal R_{c_N}$, hence the approximation 
is furnished by Proposition \ref{prop:glueannulus}. 
\item[(c)]\label{mainthmproof-minimum} If $\xi_{N}$ is a local minimum start by extending the immersion (resp. embedding)
to the component of $\overline{\mathcal R_{c_{N}+\delta_{N}}}$ that contains 
$\xi_N$; this is trivial because this component is a disk.   Make sure that the image 
lies in $X\setminus K_{r_{N+1}}$.  Since we may now reach $\mathcal R_{c_{N+1}-\delta_{N+1}}$
by attaching a finite number of annuli, the approximation $f_{N+1}$ is furnished by 
Proposition \ref{prop:glueannulus}. If $\dim(X)\geq 3$ the separation of points is not a problem.
\item[(d)]\label{mainthmproof-critical} If $\xi_N$ is not a local minimum, we may choose an initial approximation $\tilde f_{N+1}:\overline{\mathcal R_{c_N+\delta_N}}\rightarrow X$
furnished by Proposition \ref{prop:gluecritical}. Now $\mathcal R_{c_{N+1}-\delta_{N+1}}$ may be reached 
by attaching a finite number of annuli, and the approximation $f_{N+1}$ is 
furnished by Proposition \ref{prop:glueannulus}.
\item[(e)]\label{mainthmproof-limit} It is now clear that the limit $f:=\lim_{j\rightarrow\infty}f_j$ is well defined on $\mathcal R$, 
and gives us the desired immersion (resp. embedding) into $X$.  
For the embedding, the sequence $\epsilon_j$ should be modified along the way to 
avoid self intersections in the limit.
\end{enumerate}
\end{itemize}
Note that in the special case of $\mathcal{R} = \CC$ we can choose an exhaustion function $\rho : \CC \to \RR^+_0$ like $\rho(z) = |z|^2$, which is obviously strongly subharmonic and has no critical points except $z = 0$. After for the initial embedding we therefore never again encounter a critical point, hence the cases (c) and (d) do never occur and we do not need to require the connectedness of $X \setminus K_j$.
\end{proof}

\begin{remark}
Note that the construction would also work in the case that 
$\mathcal R$ admits an exhaustion function $\rho$ with finitely many critical points.  
We would start by embedding a totally real skeleton of $\mathcal R$, containing 
all critical points of $\rho$, into X, approximate on a small neighborhood using 
Mergelyan's theorem, and then proceed as above.   
\end{remark}

The main result in \cite{Andrist} by the first author characterizes certain Stein manifolds by their endomorphism semigroup and gives an application of our theorem using a properly embedded complex line in a Stein manifold:
\begin{theorem}
\label{thm:holconjugate}
Let $X$ and $Y$ be complex manifolds and $\Phi : \enmo{X} \to \enmo{Y}$ an isomorphism of semigroups of holomorphic endomorphisms.
Then there exists a unique $\varphi : X \to Y$ which is is either biholomorphic or antibiholomorphic and such that $\Phi(f) = \varphi \circ f \circ \varphi^{-1}$ if the following criteria are fulfilled:
\begin{enumerate}
\item $X$ is a Stein manifold, and
\item $X$ admits a proper holomorphic embedding $i : \CC \hookrightarrow X$.
\end{enumerate}
If the automorphism group of $X$ acts (weakly) double-transitive, it is sufficient for $\Phi$ to be an epimorphism.
\end{theorem}

From Theorem \ref{thm:embedsurface} and the preceeding result and noting that a Stein manifold with (volume) density property has a double-transitive action by Propostion \ref{prop:transitive} resp. its Corollary \ref{cor:transitive}, we immediately get the following result:

\begin{theorem}
\label{thm:densconjugate}
Let $X$ and $Y$ be complex manifolds and $\Phi : \enmo{X} \to \enmo{Y}$ an epimorphism of semigroups of holomorphic endomorphisms.
If $X$ is a Stein manifold with density or volume density property and of dimension at least $3$, then there exists a unique $\varphi : X \to Y$ which is is either biholomorphic or antibiholomorphic and such that $\Phi(f) = \varphi \circ f \circ \varphi^{-1}$.
\end{theorem}

A conjecture by Schoen and Yao \cite{SY} claimed that no proper harmonic map could exist from the unit disk onto $\RR^2$.
The conjecture was recently disproven by Alarc\'{o}n and Galv\'{e}z \cite{AG}, but a much stronger result follows easily from our main theorem: 

\begin{theorem}
\label{thm:properharmonic}
Let $\mathcal R$ be any open Riemann surface. Then $\mathcal R$
admits a proper harmonic mapping into $\RR^2$.
\end{theorem}

\begin{proof}
The Stein manifold $\CC^* \times \CC^*$ has the volume density property (with standard volume form $\frac{\dint z}{z} \wedge \frac{\dint w}{w}$), see \cite{Varolin1}.
According to Theorem \ref{thm:embedsurface} there exists a proper holomorphic immersion $(f_1, f_2) : \mathcal{R} \to \CC^* \times \CC^*$.
The map $(\log|f_1|, \log|f_2|) : \mathcal{R} \to \RR^2$ is harmonic and still proper.
\end{proof}

\end{document}